\documentclass{article}
\usepackage[utf8]{inputenc}
\usepackage[T1]{fontenc}
\usepackage{lmodern}
\usepackage{graphicx}

\usepackage{amsmath,amssymb,amsthm}
\newtheorem{theoreme}{Theorem}[section]
\newtheorem{definition}[theoreme]{Definition}
\newtheorem{remarque}[theoreme]{Remark}
\newtheorem{proposition}[theoreme]{Proposition}
\newtheorem{lemme}[theoreme]{Lemma}
\newtheorem{corollary}[theoreme]{Corollary}

\newcommand{\cp}{\mathbb{C}\mathbb{P}}

\newcommand{\cmplx}{\mathbb{C}}
\newcommand{\real}{\mathbb{R}}

\newcommand{\ellipse}{\mathcal{E}}

\newcommand{\bigaxe}{\mathcal{F}}

\newcommand{\ccenters}{\mathcal{C}} 
\newcommand{\inftyline}{\overline{\cmplx}_{\infty}}

\begin{document}

\title{On the circumcenters of triangular orbits in elliptic billiard}
\author{}
\date{}
\maketitle

\noindent Corentin FIEROBE, École Normale Supérieure de Lyon, Unité de Mathématiques Pures et Appliquées, UMR CNRS 5669, 46, allée d’Italie, 69364 Lyon Cedex 07, France

\begin{abstract}
On an elliptic billard, we study the set of the circumcenters of all triangular orbits and we show that this is an ellipse. This article follows \cite{romaskevich2014}, which proves the same result with the incenters, and \cite{glut1}, which among others introduces the theory of complex reflection in the complex projective plane. The result we present was found at the same time by Ronaldo Garcia in an article to appear in American Mathematical Monthly (no preprint available). His proof uses completely different methods of real differential calculus.
\end{abstract}

\tableofcontents

\section{Overview of the problem}

The famous great theorem of Poncelet, \textit{cf} \cite{berger90} and \cite{poncelet}, asserts that \textit{if there exists an $n$-sided polygon inscribed in a conic $\mathcal{C}$ and circumscribed about an other one $\mathcal{D}$, then there are infinitely many such polygons, and you can find such one for each point of $\mathcal{C}$ chosen to be one of its vertices}. A classical proof of it can be found in \cite{berger90}. \cite{poncGH} and \cite{schwartz} give a way to prove it using complex methods.

It has a lot of consequences (see \cite{berger90}, \cite{DragRad}), especially in billard theory, since it gives a condition to the existence of particular $n$-periodical orbits in conics. In particular, given an ellipse $\ellipse$, one can find a confocal ellipse $\gamma$ to $\ellipse$, such that each triangular orbit on $\ellipse$ is circumscribed about it; and conversely one can complete each tangent line to $\gamma$ in a triangular orbit of $\ellipse$.

We study here the set of circumcenters (the centers of the circumscribed circles) of all triangular orbits in such an elliptic billard $\ellipse$. We want to prove the following:

\begin{theoreme}
	\label{thethm}
	The set $\ccenters$ of circumcenters of all triangular orbits on an ellipse is also an ellipse.
\end{theoreme}

\begin{remarque}
\label{rmk_circle}
Theorem \ref{thethm} is obvious in the particular case where the ellipse is a circle, because then the set of circumcenters is reduced to a single point. Thus, from now on \textit{we will consider that the ellipse is not a circle}.
\end{remarque}

There are many other results similar to theorem \ref{thethm}. Dan Reznik discovered experimentally the same result for the incenters of triangular orbits (see the video \cite{reznik_youtube} and the github page \cite{reznik_github} written with Jair Koiller) and Romaskevitch (see \cite{romaskevich2014}, whose proof widely inspired ours) confirmed these observations by proving them. Tabachnikov and Schwartz, in \cite{taba}, proved that the loci of the centers of mass (and of an other particular point) of a $1$-parameter family of Poncelet $n$-gons in an ellipse is an ellipse homothetic to the previous one. They also mention that a similar result was proved by Zaslawski, Kosov and Muzafarov for the orthocenters (\cite{zaslawski}, reference from \cite{taba}). And Garcia (see \cite{ronaldo}) uses explicit calculations to prove that the loci of incenters and orthocenters of triangular orbits are ellipses, and describes them precisely.\\

Before going into details, we give here a brief summary of the proof, which is inspired by \cite{romaskevich2014} (we use the same complex methods, see further). We consider a projective complexified version of $\ccenters$, denoted by $\hat{\ccenters}$, which turns out to be an algebraic curve (consequence of Remmert proper mapping theorem, see \cite{GH} p. 34). Then we show that the intersection of the complex curve $\hat{\ccenters}$ with the foci line of the ellipse is reduced to two points, each one of them corresponding to a single triangular orbit. Further algebraic arguments on the intersection type of $\hat{\ccenters}$ with the foci line of $\ellipse$ allow to conclude that it is a conic, using Bezout theorem. It's then easy to check that $\hat{\ccenters}$ is an ellipse since its real part is bounded.

As explained, one considers the projective complex Zariski closure of the ellipse $\ellipse$ and a complexified version of $\ccenters$, $\hat{\ccenters}$. In order to define $\hat{\ccenters}$ and to prove the first statement concerning the intersection with the foci line, we study an extension of the reflection law and of the triangular orbits to complex domain, as in \cite{romaskevich2014} (which proves Proposition \ref{lemma_olga} used in this article). Complex reflection law and complex planar billiards were introduced and studied by A. Glutsyuk in \cite{glut1} and \cite{glut2}. See also \cite{glut3} where they were applied to solve the two-dimensional Tabachnikov's Commuting Billiard conjecture and a particular case of two-dimensional Plakhov's Invisibility conjecture with four reflections.

\textbf{Section \ref{sec_general_reflexion}} is devoted to the complex reflexion law and to complex orbits in a complexified ellipse: in \textbf{Subsection \ref{sec_reflection}}, we introduce the complex reflexion law; \textbf{Subsection \ref{sec_cmplxconic}} recalls some results about complexified conics; we further define what is a triangular complex orbit in \textbf{Subsection \ref{sec_orbits}}; then, in \textbf{Section \ref{sec_circumcircles}} we introduce the definition and we study properties of complex circumscribed circles to such orbits: Proposition \ref{prop_orbit_form} is the main result of this section. Finally, \textbf{Section \ref{sec_proof}} is devoted to the proof of Theorem \ref{thethm}, using previous results.

\section{Complex triangular obits on an ellipse}
\label{sec_general_reflexion}
\subsection{Complex reflection law}
\label{sec_reflection}

Considering an affine chart whose coordinates will be denoted by $(x,y)$, we have the inclusion $\real^2 \subset \cmplx^2 \subset \cp^2$, and $\cp^2 = \cmplx^2 \sqcup \inftyline$, where $\inftyline$ is the infinity line. As introduced and explained in \cite{glut1}, and studied in \cite{romaskevich2014}, the reflection law in $\real^2$ can be extended to $\cp^2$ by considering the complexified version of the canonical euclidean quadratic form 
$$dx^2+dy^2$$
which is a non degenerate quadratic form on $\cmplx^2$. In a similar way to the euclidean case, it leads to construct a notion of symmetry in $\cmplx^2$. But because it has two isotropic subspaces of dimension $1$ (namely $\cmplx (1,i)$ and $\cmplx (1,-i)$), the notion of orthogonal space, and hence of symmetrical line, is not always defined. This is the reason why one needs to adapt the notion of symmetry. A similar notion of symmetry with isotropic spaces is studied in \cite{khesin} and \cite{dragovic} in pseudo-Euclidean and pseudo-Riemannian cases.

\begin{definition}[\cite{glut1}, definition 1.2]
	A line in $\cp^2$ is said to be \textit{isotropic} if it contains either $I = [1:i:0]$ or $J = [1:-i:0]$, the isotropic points at infinity (or cyclic points), and \textit{non-isotropic} if not. (Thus, the infinity line is automatically isotropic.)
\end{definition}

\begin{definition}[\cite{glut1}, definition 2.1]
The \textit{symmetry} with respect to a line $L$ is defined by the two following points:
	\begin{itemize}
		\item the \textit{symmetry acting on $\cmplx^2$}: it is the unique non-trivial complex-isometric involution fixing the points of the line $L$, if $L$ is non-isotropic ;
		\item the \textit{symmetry acting on lines}: if $L$ is an isotropic line going through a finite point $x$, two lines $l$ and $l'$ going through $x$ are called symmetric if there are sequences $(L_n)_n$, $(l_n)_n$, $(l'_n)_n$ of lines through points $x_n$ so that $L_n$ is non-isotropic, $l_n$ and $l'_n$ are symmetric with respect to $L_n$, $l_n \to l$, $l'_n \to l'$, $L_n \to L$ and $x_n \to x$.
	\end{itemize}
\end{definition}

We recall now lemma $2.3$ \cite{glut1} which gives an idea of this notion of symmetry in the case of an isotropic line through a finite point.

\begin{lemme}[\cite{glut1}, lemma 2.3]
\label{lemma_glutsyuk}
If $L$ is an isotropic line through a finite point $x$ and $l$, $l'$ are two lines going through $x$, then $l$ and $l'$ are symmetric with respect to $L$ if and only if either $l=L$, or $l'=L$. 
\end{lemme}

This complex reflection law allows to talk about complex billiard orbits on the ellipse, as it will be done in subsection \ref{sec_orbits}. Before studying those orbits, it is necessary to introduce some geometric notions about projective conics.

\subsection{Preliminary results on complexified conics}
\label{sec_cmplxconic}

One needs to present here useful results on confocal conics. They can be found in \cite{berger90} and \cite{klein26}. \cite{glut1} also cites them in subsection 2.4. This section allows to understand some links between an ellipse and its Poncelet ellipse of triangular orbits, when both are complexified. Thus, by conic (resp. ellipse) we mean here the complex projective closure of a real regular conic (resp. ellipse). This choice of definition is due to the fact that the ellipse in which we study billard orbits is a real ellipse. Later, in order to define circumscribed circle (see section \ref{sec_circumcircles}), we will understand conics as complex conics (not just complexified real ones).

\begin{proposition}[\cite{berger90} subsection 17.4.2.1]
\label{circle_cycpoints}
A conic is a circle if and only if some of the cyclic points $I$ or $J$ belong to it. Furthermore, if a conic is a circle, then both $I$ and $J$ belong to it.
\end{proposition}

In fact, a circle has two isotropic tangent lines intersecting at its center (see the following propositions).

\begin{proposition}[\cite{berger90} subsection 17.4.3.1]
\label{tangent_cycpoints}
A focus $f$ of a conic lies in the intersection of two isotropic tangent lines to the conic.
\end{proposition}

\begin{proposition}[\cite{klein26}, p.~179]
\label{confocal_conics}
Two complexified confocal ellipses have the same tangent isotropic lines, which are four isotropic lines taken with multiplicities: one pair intersecting on a focus, and the other one - on the other focus.
\end{proposition}

This brings us to the following redefinition of the foci:

\begin{definition}[\cite{berger90} subsection 17.4.3.2]
The (complex) \textit{foci} of an ellipse are the points in $\cmplx^2$ on which its isotropic tangent lines intersect.
\end{definition}

\begin{remarque}
The complex projective closure of a real ellipse has four complex foci, including two real ones.
\end{remarque}

\begin{corollary}
\label{istropy_foci}
A conic has at most four dinstinct finite isotropic tangent lines, each two of them intersecting either at a focus, or at an isotropic point at infinity.
\end{corollary}




\subsection{Complex orbits}
\label{sec_orbits}
 
We have enough material at this stage to study complex triangular orbits. See \cite{glut1}, definition 1.3, for a more general definition of periodic orbits.

\begin{definition}[\cite{glut1}, definition 1.3]
A \textit{non-degenerate triangular complex orbit} on a complex conic $\ellipse$ is an ordered triple of points $A_1,A_2,A_3$ on $\ellipse$ so that for all $i$, one has $A_i \neq A_{i+1}$, the tangent line $T_{A_i}\ellipse$ is non isotropic, and the complex lines $A_iA_{i+1}$ and $A_iA_{i-1}$ are symmetric with respect to $T_{A_i}\ellipse$ (with the obvious convention $A_0 := A_3$ and $A_4 := A_1$). A \textit{side} of a non-degenerate orbit is a complex line $A_iA_{i+1}$.
\end{definition}

\begin{remarque}
The vertices of a non-degenerate orbit are not collinear since a line intersects the ellipse in at most two points.
\end{remarque}

\begin{remarque}
\label{rem_non_isotropic_sides}
As explained in \cite{glut2}, the reflexion with respect to a non-isotropic line permutes the isotropic directions $I$ and $J$. This argument implies that a non-degenerate triangular orbit has no isotropic side.
\end{remarque}

We will also study the limit orbits of the above defined orbits, which will be called degenerate orbits.

\begin{definition}[\cite{glut1}]
A \textit{degenerate triangular complex orbit} on a complex conic $\ellipse$ is an ordered triple of points in $\ellipse$ which is the limit of non-degenerate orbits and which is not a non-degenerate orbit. We define the \textit{sides} of a degenerate orbit as the limits of the sides of the non-degenerate orbits which converge to it.
\end{definition}

\begin{proposition}[\cite{romaskevich2014}, lemma 3.4]
\label{lemma_olga}
A degenerate orbit of an ellipse $\ellipse$ has an isotropic side $A$ which is tangent to $\ellipse$, and two coinciding non-isotropic sides $B$.
\end{proposition}

During the proof, it will be convenient to distinguish two types of orbits : the ones with no points at infinity, and the others, with at least one point at infinity:

\begin{definition}
An \textit{infinite triangular complex orbit} on a complex conic $\ellipse$ is an orbit which owns at least one vertex on the infinity line. The orbits with only finite vertices are called \textit{finite orbits}.
\end{definition}

\begin{proposition}
\label{lemma_infinite_orbit}
An infinite orbit is not degenerate, and owns exactly one vertex at infinity.
\end{proposition}

\begin{proof}
Suppose two vertices, $\alpha,\beta$, of the orbit are at infinity. Then, $\alpha \beta$ is the infinity line. But the tangent $T_{\beta}$ to the ellipse $\ellipse$ in $\beta$ is not isotropic, and the infinity line reflects to itself through the reflexion by $T_{\beta}$. Hence, the orbit is $\{\alpha,\beta\} = \inftyline \cap \ellipse$, which should be a degenerate orbit. But it cannot be a degenerate orbit by Proposition \ref{lemma_olga} since the tangent lines to its vertices $\alpha, \beta$ are not isotropic. Thus, only one vertex lies at infinity.

Therefore, if it is a degenerate orbit, it has two vertices, $\alpha,\beta$, corresponding by Proposition \ref{lemma_olga} to two sides, $A$ which is isotropic and tangent to the ellipse in $\alpha$, and $B$ which is a line going through $\alpha$ and $\beta$. Since the tangency points of isotropic tangent lines are finite, $\alpha$ is finite. Thus $\beta$ is infinite (because the orbit is supposed infinite). Then $B$ and the tangent line $T_{\beta}\ellipse$ to the ellipse in $\beta$ are collinear (since they have the same intersection point at infinity). But both are stable by the complex reflexion by $T_{\beta}$, hence $T_{\beta}\ellipse = B$ which is impossible since $B$ is not tangent to the ellipse.
\end{proof}

\section{Circumcircles and circumcenters of complex orbits}
\label{sec_circumcircles}

Here we present the last part of the required definitions, which concerns the complex circles circumscribed to triangular orbits. This part is different from the previous one, because here the considered conics are complex and not necessarily complexified versions of real conics.

\begin{definition}
A \textit{complex circle} is a regular complex conic passing through both isotropic points at infinity. Its \textit{center} is the intersection point of its tangent lines at the isotropic points.
\end{definition}

\begin{remarque}
In every sequence of complex circles, one can choose a subsequence which limits to either a circle, or a pair of isotropic finite lines, or a pair of lines from which one is being infinite and the other one is being finite, or the infinity line taken twice.
\end{remarque}

\begin{proposition}
\label{prop_circle_reg_orbit}
For a non-degenerate finite orbit, there is a unique complex circle passing through the vertices of the orbit and both isotropic points at infinity. It is called the circumscribed circle or circumcircle to the non-degenerate orbit.
\end{proposition}

\begin{proof}
Denote by $\alpha, \beta, \gamma$ the vertices of the orbit. We have to prove that no three points of $\alpha, \beta, \gamma, I, J$ are collinear. Indeed, as no vertices are on the infinity line, we only need to study two different cases:
\begin{enumerate}
    \item $\alpha, \beta, \gamma$ are not collinear because they are disctinct and they lie on the ellipse which has at most two intersection points with any line.
    \item $\alpha, \beta, I$ are not collinear or else the line $\alpha \beta$ would be isotropic. But this is impossible for a non-degenerate triangular orbit by Remark \ref{rem_non_isotropic_sides}.
    
    We then exclude all other possible combinations of two vertices of the orbit with $I$ or $J$, using the same arguments.
    
    
\end{enumerate}
\end{proof}

Let us extend this definition to degenerate orbits.

\begin{definition}
\label{def_circle_deg_orbit}
If $T$ is a degenerate or a infinite orbit, its \textit{circumscribed circle} is a limit of the circumscribed circles of non-degenerate finite orbits converging to $T$. The \textit{circumcenter} of $T$ is defined in a similar way.
\end{definition}

Even if they are called \textit{circle}, the circumscribed circles to a degenerate or infinite orbit can degenerate into pairs of lines, as described below.

\begin{proposition}
\label{prop_types_circle}
The possible cases for the circumscribed circle to a degenerate or infinite orbit are the following :
\begin{enumerate}
	\item a regular circle ;
	\item a pair of isotropic  non-parallel finite lines ; the corresponding center lies on their intersection ; 
	\item the infinite line and a finite line $d$ ; the center $c$ lies on the infinity line and represents a direction which is orthogonal to $d$ ; 
	\item the infinity line (taken twice).
\end{enumerate}
\end{proposition}

\begin{proof}
The equation of a regular circle $\mathcal{D}$ is of the form
$$a(x^2+y^2)+pxz+qyz+rz^2 = 0$$
where $a, p,q,r\in\cmplx$, $a\neq 0$ and $4a r \neq p^2+ q^2$. Both isotropic tangent lines to $\mathcal{D}$ have equations $2a(x \pm iy) +(p\pm i q)z = 0$, whose intersection is $c = (p:q:-2a)$, which is the center of $\mathcal{D}$ by definition.

If we take a limit of regular circles, the equation of the limit circle is of the same type, that is
$$a(x^2+y^2)+pxz+qyz+rz^2 = 0$$
but maybe with $a=0$ or $4a r = p^2+ q^2$. And the center $c$ is still of coordinates $(q:p:-2a)$. 

If $a=0$, the limit circle is the union of the infinity line ($z=0$) and the line $d$ of equation $px+qy+rz = 0$. The line $d$ is finite if and only if $(p,q)\neq 0$, and in this case it has direction $(q,-p)$. Since $c = (p:q:0)$, the direction represented by $c$ is orthogonal to $d$. If $d$ is infinite, the limit circle is the (double) infinity line. Note that in this case the center can be an arbitrary point.

If $a\neq 0$, but $4a r = p^2+ q^2$, the equation of the limit circle becomes
$$\left(x+\frac{p}{2a}z\right)^2+\left(y+\frac{q}{2a}z\right)^2=0$$
which is the equation of two isotropic non collinear lines intersecting at the point 
$(-\frac{p}{2a}:-\frac{q}{2a}:1) = (p:q:-2a) = c$.

If $a\neq 0$ and $4a r \neq p^2+ q^2$, the limit circle is regular.
\end{proof}

Now let us find which triangular orbits have their center on the line of real foci of $\ellipse$. 

\begin{proposition}
\label{prop_orbit_form}
A complex triangular orbit of circumcenter lying on the foci line is finite, non-degenerate, symmetric with respect to the real foci line of $\ellipse$, and has a vertex on it.
\end{proposition}

\begin{proof}
Let $T$ be a triangular orbit with a circumscribed circle $C$ having a center $c$ on the real foci line of $\ellipse$.

\textbf{First case : Suppose $T$ is finite and non-degenerate.} We follow the arguments of Romaskevich \cite{romaskevich2014} who treated the similar case for incenters. Indeed, in this case two vertices of $T$ should lie on a line orthogonal to the real foci line: otherwise, by symmetry of $\ellipse$ and $C$ with respect to the foci line, the intersection $\ellipse\cap C$ would count six points, which is impossible since $\ellipse$ is not a circle. Finally, the remaining vertex has to be on the foci line, or else we could find two distinct orbits sharing a common side, which is impossible by definition of the reflexion law with respect to non-isotropic lines.

\textbf{Second case : Suppose $T$ is infinite.} Then the infinity line cuts $C$ in three distinct points, hence $C$ is degenerate. By Proposition \ref{prop_types_circle}, $C$ contains the infinity line. Since $T$ has only one infinite vertex $\alpha$ by Proposition \ref{lemma_infinite_orbit}, and two other finite vertices $\beta,\gamma$, the other line $d$ is not the infinity line. Again by Proposition \ref{prop_types_circle}, the center is infinite and represents the orthogonal direction to $d$. Since it is on the real foci line, the latter is orthogonal to $d$. Then $d$ does not contains $\alpha$, or else $d$ would be infinite by symmetry with respect to the foci line. We have $d=\beta\gamma$ is a side of $T$, $\alpha\notin d$ and by the same symmetry argument as in the first case $\alpha$ should belong to the real foci line. But this is impossible since the latter cuts $\ellipse$ in only two finite points.

\textbf{Last case : Suppose $T$ is degenerate.} Then $C$ cannot be a regular circle, otherwise the latter would be tangent to $\ellipse$ in a point of isotropic tangency (by Proposition \ref{lemma_olga}): this would imply that this point of isotropic tangency is $I$ or $J$, which is impossible since they do not belong to $\ellipse$, assumed not to be a circle.

The circumcircle $C$ cannot be the union of the infinity line and another line $d$. Otherwise, by the same arguments as in the second case, this line would be othogonal to the real foci line. Since $T$ is finite (Proposition \ref{lemma_infinite_orbit}), $d$ goes through its both vertices, implying that both are points of isotropic tangency of $\ellipse$. But this cannot happen for a degenerate triangular orbit.

Finally suppose $C$ is  the union of two isotropic lines having different directions.

\begin{lemme}
\label{lemma_isotropic_circles}
Let $C_n$ be a sequence of circles going through two distinct points $M_n$ and $N_n$ of $\ellipse$ converging to the same finite point $\alpha$. Suppose $C_n$ has a center $c_n$ converging to a finite point $c\neq\alpha$. Then the line $c\alpha$ is orthogonal to the line $T_{\alpha}\ellipse$.
\end{lemme}

\begin{proof}
The tangent line to $C_n$ at $M_n$ is orthogonal to the line $M_nc_n$ hence the same is true for their limits. The limit of $T_{M_n}C_n$ is obviously the limit of the line $M_nN_n$. Since $M_n$ and $N_n$ are on $\ellipse$, the line $M_nN_n$ also converges to the tangent line $T_{\alpha}\ellipse$. Hence $T_{\alpha}\ellipse$ is orthogonal to $\alpha c$.
\end{proof}

Thus if $\alpha$ is a vertex of isotropic tangency of the orbit, Lemma \ref{lemma_isotropic_circles} implies that $\alpha c$ is orthogonal to $T_{\alpha}\ellipse$, hence $\alpha c = T_{\alpha}\ellipse$ since the latter is isotropic. Since both isotropic lines constituing the circle go through $c$, one of them is $T_{\alpha}\ellipse$, and they are both tangent to $\ellipse$ by symmetry with respect to the real foci line. Thus the other vertex of $T$ is a point of isotropic tangency of $\ellipse$, which is not possible by the previous arguments (such an orbit is not closed).
\end{proof}

\section{Proof of Theorem \ref{thethm}}
\label{sec_proof}

We reall that $\ellipse$ is a complexified ellipse, which we will identify with $\overline{\cmplx}$. Denote by $\gamma$ the real ellipse inscribed in all triangular real periodic orbits. We use the same notation $\gamma$ for its complexified version. 

Consider the Zariski closure $\mathcal{T}$ of the set of real triangular orbits (which are circumscribed about $\gamma$). Let $\mathcal{T}_3$ denote the set of triangles with vertices in $\ellipse$ that are circumscribed about $\gamma$. It is a Zariski closed set of $\ellipse^3\simeq\left(\cp^1\right)^3$ that contains the real orbits and can be identified with the set of pairs $(A,L)$, where $A$ is a point of the complexified ellipse $\ellipse$ and $L$ is a line through $A$ that is tangent to $\gamma$. The set of the above  pairs $(A,L)$ is identified with an elliptic curve, and each pair extends to a circumscribed triangle as above, see the complex Poncelet Theorem and its proof in \cite{poncGH} for more details. Hence $\mathcal{T}_3$ is an irreducible algebraic curve. Each triangle in $\mathcal{T}$ is circumscribed about $\gamma$, by definition and since this is true for the real triangular orbits and the tangency condition of the edges with $\gamma$ is algebraic. Thus $\mathcal{T}\subset\mathcal{T}_3$. Hence $\mathcal{T}=\mathcal{T}_3$, by definition and since the curve of real triangular orbits (which is contained in $\mathcal{T}$) is Zariski dense in $\mathcal{T}_3$ (irreducibility). Now the set $\hat{\mathcal{T}}\subset\mathcal{T}$ of complex non-degenerate triangular orbits circumscribed about the Poncelet ellipse $\gamma$ is a subset of $\mathcal{T}_3 = \mathcal{T}$, Zariski open in $\mathcal{T}$ (because $\mathcal{T}\setminus\hat{\mathcal{T}}$ is defined by polynomial equations). Note that $\mathcal{T}\setminus\hat{\mathcal{T}}$ is finite (since it is a proper Zariski closed subset of $\mathcal{T}$), and $\hat{\mathcal{T}}$ is dense in $\mathcal{T}$ for the usual topology. Thus the analytic map $\phi : \hat{\mathcal{T}} \to \cp^2$ which assigns to a non-degenerate orbit its circumcenter can be extended to a holomorphic map $\mathcal{T}\to\cp^2$, being a rational map. And by Remmert proper mapping theorem (see \cite{GH}), its image denoted by $\hat{\ccenters}$ is an irreducible analytic curve of $\cp^2$, hence it is an irreducible algebraic curve by Chow theorem (see \cite{GH}). 


Let us show that $\hat{\ccenters}$ is a conic, using Bezout theorem and studying its intersection with the real foci line of $\ellipse$. In fact, we already know two distinct points lying on this intersection: the circumcenters $c_1$ and $c_2$ of both triangular real orbits $T_1$ and $T_2$ circumscribed about Poncelet's ellipse $\gamma$ and having a vertex on the foci line.

\begin{lemme}
\label{lemma_nb_intersection}
The foci line of the ellipse intersects $\hat{\ccenters}$ in only $c_1$ and $c_2$ which are distinct, and for each $i$ the only triangular orbit of $\mathcal{T}$ having $c_i$ as a circumcenter is $T_i$.
\end{lemme}

\begin{proof}
Take a point $c$ of $\hat{\ccenters}$ lying on the foci line. Then by Proposition \ref{prop_orbit_form}, an orbit of center $c$ is finite, non-degenerate, and has a vertex on the foci line. If this orbit is in $\mathcal{T}$, it is circumscribed about $\gamma$. One of its vertices lies on the foci line, hence coincides with a vertex of some $T_i$. Hence it is $T_1$ or $T_2$, otherwise we could find a number strictly greater than two of tangent lines to $\gamma$ going through a vertex of $\ellipse$. Furthermore, if $c_1=c_2$, the circumcircle of $T_1$ would be the same as the one of $T_2$ by symmetry, and $\ellipse$ would share six dictinct points with the former, which is impossible. The result follows.
\end{proof}

\begin{theoreme}
\label{lemme_intersection}
The set $\hat{\ccenters} \subset \cp^2$ is an ellipse.
\end{theoreme}

\begin{proof}
Let us show that $c_1$ is a regular point of $\hat{\ccenters}$, and that the latter intersects the foci line transversally. Fix an order on the vertices of $T_1$ and consider the germ $(\mathcal{T},T_1)$. The latter is irreducible (because parametrized by $\gamma$), hence the germ $(V,c_1)\subset(\hat{\ccenters},c_1)$ defined as $\phi(\mathcal{T},T_1)$ is also irreducible. By Lemma \ref{lemma_nb_intersection}, any other irreducible component $V'$ of $(\hat{\ccenters},c_1)$ is parametrized locally by $\phi$ and a germ $(\mathcal{T}, T_1')$, where $T_1'$ is obtained from $T_1$ by a permutation of its vertices. Thus $V'=V$ since $\phi$ doesn't change by permutation of the vertices of the orbits: $(\hat{\ccenters},c_1)$ is irreducible. 

We fix a local biholomorphic parametrization $P(t)$ of the complexified ellipse $\ellipse$, so that $P_0 = P(0)$ is a vertex of the real ellipse $\ellipse$ that is also a vertex of the real triangular orbit $T_1$. This gives local parametrizations of the orbits $T(P)$ whose first vertex is $P$ and of their circumcenters $c(t) = \phi(T(P(t)))$. We restrict $P$ to the curve $P(t)$ parametrizing the real points of $\ellipse$. We can suppose that $P(t)$ and $P(-t)$ are symmetric with respect to $\bigaxe$. Write $r(t) = |P(t)c(t)|$ for the radius of the circumscribed circle to $T(t)$. Thus we have $c(0)=\phi(T_1)=c_1$, and we need to show that $c'(0) \neq 0$ and that $c'(0)$ has not the same direction as the line of real foci of $\ellipse$. 

First, we have $r(t) = r(-t)$ by symmetry, and $r$ is smooth around $0$ since $P(0) \neq c(0)$. Thus, $r'(0) = 0$ which implies that $c'(0)-P'(0)$ is orthogonal to the foci line. But $P'(0)$ is already orthogonal to the foci line, hence the same hold for $c'(0)$. It's then enough
to show that $c'(0) \neq 0$.

Suppose the contrary, i.e. $c'(0) = 0$. We use again $r'(0) = 0$. If we denote by $Q(t)$ one of the other vertices of $T(t)$ and $Q_0=Q(0)$, then since also $r(t) = |Q(t)c(t)|$, the equality $r'(0) = 0$ gives that the line $Q_0c_1$ is orthogonal to $T_{Q_0}\ellipse$. It means that the circumscribed circle $\mathcal{D}$ to $T_1$ has the same tangent line in $Q_0$ as $\ellipse$. Since this is also true in $P_0$ and in the third point of $T_1$ (same proof), we get that $\ellipse$ and $\mathcal{D}$ have three common points with the same tangent lines, which means that $\ellipse$ is a circle. But this case was excluded at the beginning (remark \ref{rmk_circle}).

Hence $c'(0) \neq 0$ and $c'(0)$ is orthogonal to the line of real foci. The proof is the same for $c_2$. Hence by Bezout theorem, $\hat{\ccenters}$ is an ellipse.
\end{proof}

\section{Acknowledgments}

This article could not be possible without the precious help, advices and support of Alexey Glutsyuk and Olga Romaskevich, who also suggested me to work on this topic. This article is an adaptation of the proof of Olga, and many ideas come from her article, \cite{romaskevich2014}. I am also really grateful to Anastasia Kozyreva for her help.


\end{document}